\documentclass[a4paper,twoside,11pt,reqno]{amsart}
\usepackage{fullpage}
\usepackage[T1]{fontenc}
\usepackage[utf8]{inputenc}

\usepackage{hyperref}
\usepackage[slantedGreeks, partialup, noDcommand]{kpfonts}
\usepackage{graphicx}
\usepackage[mathscr]{euscript}

\hypersetup{ocgcolorlinks=true,allcolors=testc}
\hypersetup{
     colorlinks   = true,
     citecolor    = black
}
\hypersetup{linkcolor=black}
\hypersetup{urlcolor=black}

\newcommand{\eqdef}{\stackrel{\mathrm{def}}{=}}

\newtheorem{theorem}{Theorem}

\newtheorem{question}[theorem]{Question}
\newtheorem{proposition}[theorem]{Proposition}
\newtheorem{corollary}[theorem]{Corollary}

\newtheorem{remark}[theorem]{Remark}
\newtheorem{definition}[theorem]{Definition}
\newtheorem*{definition*}{Definition}

\newcommand*\diff{\mathop{}\!\mathrm{d}}
\newcommand{\R}{\mathbb{R}} 

\newcommand{\N}{\mathbb{N}}

\newcommand{\e}{\varepsilon}

\usepackage{dutchcal}

\newcommand{\MM}{\mathcal{M}}
\newcommand{\NN}{\mathcal{N}}

\newcommand{\mb}{\mathbb}
\newcommand{\ms}{\mathscr}

\title{\mbox{$\varepsilon$-isometric dimension reduction for incompressible subsets of $\ell_p$}}

\author{Alexandros Eskenazis}
\address{CNRS, Institut de Math\'ematiques de Jussieu, Sorbonne Universit\'e, France and Trinity College, University of Cambridge, UK.}
\email{alexandros.eskenazis@imj-prg.fr, ae466@cam.ac.uk}

\thanks{The author was supported by a Junior Research Fellowship from Trinity College, Cambridge. A conference version of this article will be presented in SoCG 2022.}
\begin{document}

\maketitle
\vspace{-3mm}

\begin{abstract}
Fix $p\in[1,\infty)$, $K\in(0,\infty)$ and a probability measure $\mu$. We prove that for every \mbox{$n\in\N$}, $\e\in(0,1)$ and $x_1,\ldots,x_n\in L_p(\mu)$ with $\big\| \max_{i\in\{1,\ldots,n\}} |x_i| \big\|_{L_p(\mu)} \leq K$, there exists $d\leq \frac{32e^2 (2K)^{2p}\log n}{\e^2}$ and vectors $y_1,\ldots, y_n \in \ell_p^d$ such that
$$\forall \ i,j\in\{1,\ldots,n\}, \qquad \|x_i-x_j\|^p_{L_p(\mu)}- \e \leq \|y_i-y_j\|_{\ell_p^d}^p  \leq \|x_i-x_j\|^p_{L_p(\mu)}+\e.$$
Moreover, the argument implies the existence of a greedy algorithm which outputs $\{y_i\}_{i=1}^n$ after receiving $\{x_i\}_{i=1}^n$ as input. The proof relies on a derandomized version of Maurey's empirical method (1981) combined with a combinatorial idea of Ball (1990) and a suitable change of measure. Motivated by the above embedding, we introduce the notion of $\varepsilon$-isometric dimension reduction of the unit ball ${\bf B}_E$ of a normed space $(E,\|\cdot\|_E)$ and we prove that ${\bf B}_{\ell_p}$ does not \mbox{admit $\varepsilon$-isometric dimension reduction by linear operators for any value of $p\neq2$.}
\end{abstract}

\bigskip

{\footnotesize

}


\section{Introduction}


\subsection{Metric dimension reduction} Using standard terminology from metric embeddings (see \cite{Ost13}), we say that a mapping between metric spaces $f:(\MM,d_\MM)\to(\NN,d_\NN)$ is a bi-Lipschitz embedding with distortion at most $\alpha\in[1,\infty)$ if there exists a scaling factor $\sigma\in(0,\infty)$ such that
\begin{equation}
\forall \ x,y\in\MM,\qquad \ \sigma d_\MM(x,y) \leq d_\NN\big(f(x),f(y)\big) \leq \alpha \sigma d_\MM(x,y).
\end{equation}
Throughout this paper, we shall denote by $\ell_p^d$ the linear space $\R^d$ equipped with the $p$-norm,
\begin{equation}
\forall \ a=(a_1,\ldots,a_d)\in\R^d, \qquad \|a\|_{\ell_p^d} = \Big( \sum_{i=1}^d |a_i|^p\Big)^{1/p}.
\end{equation}
The classical Johnson--Lindenstrauss lemma \cite{JL84} asserts that if $(\ms{H},\|\cdot\|_{\ms{H}})$ is a Hilbert space and $x_1,\ldots,x_n\in\ms{H}$, then for every $\e\in(0,1)$ there exist $d\leq \tfrac{C\log n}{\e^2}$ \mbox{and $y_1,\ldots,y_n\in\ell_2^d$ such that}
\begin{equation}
\forall \ i,j\in\{1,\ldots,n\}, \qquad \|x_i-x_j\|_{\ms{H}} \leq \|y_i-y_j\|_{\ell_2^d} \leq (1+\e)\cdot \|x_i-x_j\|_\ms{H},
\end{equation}
where $C\in(0,\infty)$ is a universal constant. In the above embedding terminology, the Johnson--Lindenstrauss lemma states that for every $\e\in(0,1)$, $n\in\N$ and $d\geq\tfrac{C\log n}{\e^2}$, any $n$-point subset of Hilbert space admits a bi-Lipschitz embedding into $\ell_2^d$ with distortion at most $1+\e$. In order to prove their result, Johnson and Lindenstrauss introduced in \cite{JL84} the influential random projection method that has since had many important applicatons in metric geometry and theoretical computer science and kickstarted the field of \emph{metric dimension reduction} (see the recent survey \cite{Nao18} of Naor) which lies at the intersection of those two subjects. 

Following \cite{Nao18}, we say that an infinite dimensional Banach space $(E,\|\cdot\|_E)$ admits bi-Lipschitz dimension reduction if there exists $\alpha = \alpha(E)\in[1,\infty)$ such that for every $n\in\N$, there exists $k_n=k_n(E,\alpha)\in\N$ satisfying
\begin{equation}
\lim_{n\to\infty} \frac{\log k_n}{\log n} = 0
\end{equation}
and such that any $n$-point subset $\ms{S}$ of $E$ admits a bi-Lipschitz embedding with distortion at most $\alpha$ in a finite-dimensional linear subspace $F$ of $E$ with $\mathrm{dim}F\leq k_n$. The only non-Hilbertian space that is known to admit bi-Lipschitz dimension reduction is the 2-convexification of the classical Tsirelson space, as proven by Johnson and Naor in \cite{JN10}. Turning to negative results, Matou{\v s}ek proved in \cite{Mat96} the impossibility of bi-Lipschitz dimension reduction in $\ell_\infty$, whereas Brinkman and Charikar \cite{BC05} (see also \cite{LN04} for a shorter proof) constructed an $n$-point subset of $\ell_1$ which does not admit a bi-Lipschitz embedding into any $n^{o(1)}$-dimensional subspace of $\ell_1$. Their theorem was recently refined by Naor, Pisier and Schechtman \cite{NPS20} who showed that the same $n$-point subset of $\ell_1$ does not embed into any $n^{o(1)}$-dimensional subspace of the trace class $\mathsf{S}_1$ (see also the striking recent work \cite{RV20} of Regev and Vidick, where the impossibility of polynomial almost isometric dimension reduction in $\mathsf{S}_1$ is established). We refer to \cite[Theorem~16]{Nao18} for a summary of the best known bounds quantifying the aforementioned qualitative statements. Despite the lapse of almost four decades since the proof of the Johnson--Lindenstrauss lemma, the following natural question remains stubbornly open.

\begin{question} \label{q:bl}
For which values of $p\notin\{1,2,\infty\}$ does $\ell_p$ admit bi-Lipschitz dimension reduction?
\end{question}


\subsection{Dimensionality and structure} An important feature of the formalism of bi-Lipschitz dimension reduction in a Banach space $E$ is that both the distortion $\alpha(E)$ of the embedding and the dimension $k_n(E,\alpha)$ of the target subspace $F$ are independent of the given $n$-point subset $\ms{S}$ of $E$. Nevertheless, there are instances in which one can construct delicate embeddings whose distortion or the dimension of their targets depends on subtle geometric parameters of $\ms{S}$. For instance, we mention an important theorem of Schechtman \cite[Theorem~5]{Sch06} (which built on work of Klartag and Mendelson \cite{KM05}) who constructed a linear embedding of an arbitrary subset $\ms{S}$ of $\ell_2$ into any Banach space $E$ whose distortion depends only on the Gaussian width of $\ms{S}$ and the $\ell$-norm of the identity operator $\mathsf{id}_E:E\to E$. In the special case that $E$ is a Hilbert space, a substantially richer family of such embeddings was devised in \cite{LMPV17}.

Let $\mu$ be a probability measure. For a subset $\ms{S}$ of $L_p(\mu)$, we shall denote
\begin{equation}
\ms{I}(\ms{S}) \eqdef \big\|\max_{x\in \ms{S}}|x|\big\|_{L_p(\mu)}
\end{equation}
and we will say that $\ms{S}$ is $K$\emph{-incompressible}\footnote{The terminology is borrowed by the standard use of the term ``incompressible vector'' from random matrix theory, which refers to points on the unit sphere of $\mb{R}^n$ which are far from the coordinate vectors $e_1,\ldots,e_n$.} if $\ms{I}(\ms{S})\leq K$. The main contribution of the present paper is the following dimensionality reduction theorem for incompressible subsets of $L_p(\mu)$ which, in contrast to all the results discussed earlier, is valid for \emph{any} value of $p\in[1,\infty)$.

\begin{theorem}[$\varepsilon$-isometric dimension reduction for incompressible subsets of $L_p(\mu)$] \label{thm:main}
Fix parameters $p\in[1,\infty)$, $n\in\N$, $K\in(0,\infty)$ and let $\{x_i\}_{i=1}^n$ be a $K$-incompressible family of vectors in $L_p(\mu)$ for some probability measure $\mu$. Then for every $\e\in(0,1)$, there exists $d\in\N$ with $d\leq\tfrac{32e^2(2K)^{2p}\log n}{\e^2}$ and points $y_1,\ldots,y_n\in\ell_p^d$ such that
\begin{equation} \label{eq:main}
\forall \ i,j\in\{1,\ldots,n\}, \qquad \|x_i-x_j\|^p_{L_p(\mu)}- \e \leq \|y_i-y_j\|_{\ell_p^d}^p  \leq \|x_i-x_j\|^p_{L_p(\mu)}+\e.
\end{equation}

\end{theorem}

Besides the appearance of the incompressibility parameter $K$ in the bound for the dimension $d$ of the target space, Theorem \ref{thm:main} differs from the Johnson--Lindenstrauss lemma in that the error in \eqref{eq:main} is \emph{additive} rather than \emph{multiplicative}. Recall that a map between metric spaces $f:(\MM,d_\MM)\to(\NN,d_\NN)$ is called an $\e$-isometric embedding if
\begin{equation}
\forall \ x,y\in\MM, \qquad \big| d_\NN\big(f(x),f(y)\big) - d_\MM(x,y)\big| \leq \e.
\end{equation}
Embeddings with additive errors occur naturally in metric geometry and, more specifically, in metric dimension reduction (see e.g.~\cite[Section~9.3]{Ver18}). We mention for instance a result \cite[Theorem~1.5]{PV14} of Plan and Vershynin who showed that any subset $\ms{S}$ of the unit sphere in $\ell_2^n$ admits a $\delta$-isometric embedding into the $d$-dimensional Hamming cube $(\{-1,1\}^d,\|\cdot\|_1)$, where $d$ depends polynomially on $\delta^{-1}$ and the Gaussian width of $\ms{S}$. In the above embedding terminology and in view of the elementary inequality $|\alpha-\beta| \leq |\alpha^p-\beta^p|^{1/p}$ which holds for every $\alpha,\beta>0$, Theorem \ref{thm:main} asserts that any $n$-point $K$-incompressible subset of $L_p(\mu)$ admits an $\e^{1/p}$-isometric embedding into $\ell_p^d$ for the above choice of dimension $d$. For further occurences of $\e$-isometric embeddings in the dimensionality reduction and compressed sensing literatures, we refer to \cite{PV14,Jac15,Jac17,LMPV17,Ver18,BG18} and the references therein.


\subsection{Method of proof} A large part of the (vast) literature on metric dimension reduction focuses on showing that a typical low-rank linear operator chosen randomly from a specific ensemble acts as an approximate isometry on a given set $\ms{S}$ with high probability. For subsets $\ms{S}$ of Euclidean space, this principle has been confirmed for random projections \cite{JL84,FM88,DG03,Nao18}, matrices with Gaussian \cite{Gor88,IM99,Sch06}, Rademacher \cite{AV99,Ach03} and subgaussian \cite{KM05,IN07,Dir16,LMPV17} entries, randomizations of matrices with the RIP \cite{KW11} as well as more computationally efficient models \cite{Mat08,AC09,AL13,KN14,BDN15} which are based on sparse matrices. Beyond its inherent interest as an $\ell_p$-dimension reduction theorem (albeit, for specific configurations of points),  Theorem \ref{thm:main} also differs from the aforementioned works in its method of proof. The core of the argument, rather than sampling from a random matrix ensemble, relies on Maurey's empirical method \cite{Pis81} (see Section \ref{sec:maurey}) which is a dimension-free way to approximate points in bounded convex subsets of Banach spaces by convex combinations of extreme points with prescribed length. An application of the method to the positive cone of $L_p$-distance matrices (the use of which in this context is inspired by classical work of Ball \cite{Bal90}) equipped with the supremum norm allows us to deduce (see Proposition \ref{prop:bounded}) \mbox{the conclusion of Theorem \ref{thm:main} under the stronger assumption that}
\begin{equation} \label{eq:strong-ass}
K\geq\max_{i\in\{1,\ldots,n\}} \|x_i\|_{L_\infty(\mu)}.
\end{equation}
While Maurey's empirical method is an a priori existential statement that is proven via the probabilistic method, recent works (see \cite{Bar18,Iva21}) have focused on derandomizing its proof for specific Banach spaces. In the setting of Theorem \ref{thm:main}, we can use these tools to show (see Corollary \ref{cor:alg}) that there exists a greedy algorithm which receives as input the high-dimensional data $\{x_i\}_{i=1}^n$ and produces as output the low-dimensional points $\{y_i\}_{i=1}^n$. Finally, using a suitable change of measure  \cite{Mau74} (see Section \ref{sec:fact}) we are able to relax the stronger assumption \eqref{eq:strong-ass} to that of $K$-incompressibility and derive the conclusion of Theorem \ref{thm:main}. Finally, we emphasize that, in contrast to most of the dimension reduction algorithms (randomized or not) discussed earlier, the one which gives Theorem \ref{thm:main} is not \emph{oblivious} but is rather tailored to the specific configuration of points $\{x_i\}_{i=1}^n$ as it relies on the use of Maurey's empirical method.


\subsection{$\e$-isometric dimension reduction} Given two moduli $\omega,\Omega:[0,\infty)\to[0,\infty)$, we say (following \cite{Nao18}) that a Banach space $(E,\|\cdot\|_E)$ admits metric dimension reduction with moduli $(\omega,\Omega)$ if for any $n\in\N$ there exists $k_n=k_n(E)\in\N$ with $k_n=n^{o(1)}$ as $n\to\infty$ such that for any $x_1,\ldots,x_n\in E$, there exists a subspace $F$ of $E$ with $\mathrm{dim}F\leq k_n$ and $y_1,\ldots,y_n \in F$ satisfying
\begin{equation} \label{eq:mdromega}
\forall \ i,j\in\{1,\ldots,n\},\qquad \omega(\|x_i-x_j\|_E) \leq \|y_i-y_j\|_E \leq \Omega(\|x_i-x_j\|_E).
\end{equation}
In view of Theorem \ref{thm:main}, we would be interested in formulating a suitable notion of dimension reduction via $\e$-isometric embeddings which\mbox{ would be fitting to the moduli appearing in \eqref{eq:main}.}

\begin{remark}
Let $a,b\in(0,\infty)$, suppose that $\omega,\Omega:[0,\infty)\to[0,\infty)$ are two moduli satisfying
\begin{equation}
\lim_{t\to\infty} \frac{\omega(t)}{t} = a \qquad \mbox{and} \qquad \lim_{t\to\infty} \frac{\Omega(t)}{t}=b
\end{equation}
and that the Banach space $(E,\|\cdot\|_E)$ admits metric dimension reduction with moduli $(\omega,\Omega)$. Fix $n\in\N$ and $x_1,\ldots,x_n\in E$. Applying the assumption \eqref{eq:mdromega} to the points $sx_1,\ldots,sx_n$ where $s>\!\!\!>1$, we deduce that there exist points $y_1(s),\ldots,y_n(s)$ in a $k_n$-dimensional subspace $F(s)$ of $E$ such that
\begin{equation} \label{eq:mdromega2}
\forall \ i,j\in\{1,\ldots,n\},\qquad \omega(s\|x_i-x_j\|_E) \leq \big\|y_i(s)-y_j(s)\big\|_E \leq \Omega(s\|x_i-x_j\|_E).
\end{equation}
For any $\eta\in(0,1)$, we can then choose $s$ large enough (as a function of $\eta$ and the $x_i$) such that
\begin{equation} \label{eq:mdromega3}
\forall \ i,j\in\{1,\ldots,n\},\qquad (1-\eta)a\|x_i-x_j\|_E \leq \frac{\|y_i(s)-y_j(s)\|_E}{s} \leq (1+\eta)b\|x_i-x_j\|_E.
\end{equation}
Therefore, we conclude that $E$ also admits bi-Lipschitz dimension reduction (with distortion $b/a$).
\end{remark}

This simple scaling argument suggests that any reasonable notion of $\e$-isometric dimension reduction can differ from the corresponding bi-Lipschitz theory only in small scales, thus motivating the following definition. We denote by ${\bf B}_E$ the unit ball of a normed space $(E,\|\cdot\|_E)$.
\begin{definition} [$\e$-isometric dimension reduction]
Fix $\e\in(0,1)$, $r\in(0,\infty)$ and let $(E,\|\cdot\|_E)$ be an infinite-dimensional Banach space. We say that ${\bf B}_E$ admits $\e$-isometric dimension reduction with power $r$ if for every $n\in\N$ there exists $k_n=k_n^r(E,\e)\in\N$ with $k_n=n^{o(1)}$ as $n\to\infty$ for which the following condition holds. For every $n$ points $x_1,\ldots,x_n\in {\bf B}_E$ there exists a linear subspace $F$ of $E$ with $\mathrm{dim}F\leq k_n$ and points $y_1,\ldots,y_n\in F$ satisfying
\begin{equation} \label{eq:qidr}
\forall \ i,j\in\{1,\ldots,n\},\qquad \|x_i-x_j\|_E^r - \e \leq \|y_i-y_j\|_E^r \leq \|x_i-x_j\|_E^r +\e.
\end{equation}
\end{definition}

The fact that the whole space $\ell_2$ admits $\e$-isometric dimension reduction with $r=1$ and corresponding target dimension $k_n^1(\ell_2,\e)\lesssim \tfrac{\log n}{\e^2}$ follows from the additive version of the Johnson--Lindenstrauss lemma, first proven by Liaw, Mehrabian, Plan and Vershynin \cite{LMPV17} (see also \cite[Proposition~9.3.2]{Ver18}). In Corollary \ref{cor:l2} we obtain the same conclusion for its unit ball ${\bf B}_{\ell_2}$ with a slightly weaker bound for the target dimension using our Theorem \ref{thm:main}. 


It is clear from the definitions that if a Banach space $E$ admits bi-Lipschitz dimension reduction with distortion $\tfrac{1+\e}{1-\e}$, where $\e\in(0,1)$, then ${\bf B}_E$ admits $2\e$-isometric dimension reduction with power $r=1$. The $\e$-isometric analogue of Question \ref{q:bl} deserves further investigation.

\begin{question} \label{q:qi}
For which values of $p\neq2$ does ${\bf B}_{\ell_p}$ admit $\e$-isometric dimension reduction?
\end{question}

Even though the $K$-incompressibility assumption of Theorem \ref{thm:main} may a priori seem restrictive, it is satisfied for \emph{most} configurations of points in ${\bf B}_{\ell_p}$. Suppose that $n,N\in\N$ such that $N$ is polynomial\footnote{This relation between the parameters $n,N$ is natural as any $n$-point subset of $\ell_p$ embeds isometrically in $\ell_p^N$ with $N=\binom{n}{2}+1$ by Ball's isometric embedding theorem \cite{Bal90}.} in $n$. Then, standard considerations (see Remark \ref{rem:random}) show that with high probability, a uniformly chosen $n$-point subset $\ms{S}$ of $N^{1/p}{\bf B}_{\ell_p^N}$ is $O(\log n)^{1/p}$-incompressible. 


\subsection{$\e$-isometric dimension reduction by linear maps} A close inspection of the proof of Theorem \ref{thm:main} (see Remark \ref{rem:mapislinear}) reveals that in fact the low-dimensional points $\{y_i\}_{i=1}^n$ can be realized as images of the initial data $\{x_i\}_{i=1}^n$ under a carefully chosen linear operator. Nevertheless, we will show that for any $p\neq2$ and $n$ large enough, there exist an $n$-point subset of ${\bf B}_{\ell_p}$ whose image under any fixed linear $\e$-isometric embedding has rank which is linear in $n$. In fact, we shall prove the following more general statement which refines a theorem that Lee, Mendel and Naor proved in \cite{LMN05} for bi-Lipschitz embeddings.

\begin{theorem} [Impossibility of linear dimension reduction in ${\bf B}_{\ell_p}$] \label{thm:no-Lp}
Fix $p\neq2$ and two moduli $\omega,\Omega:[0,\infty)\to[0,\infty)$ with $\omega(1)>0$. For arbitrarily large $n\in\N$, there exists an $n$-point subset $\ms{S}_{n,p}$ of ${\bf B}_{\ell_p}$ such that the following holds. If $T:\mathrm{span}(\ms{S}_{n,p})\to\ell_p^d$ is a linear operator satisfying
\begin{equation}
\forall \ x,y\in\ms{S}_{n,p},\qquad \omega(\|x-y\|_{\ell_p}) \leq \|Tx-Ty\|_{\ell_p^d} \leq \Omega(\|x-y\|_{\ell_p}),
\end{equation}
then $d\geq \left(\tfrac{\omega(1)}{\Omega(1)}\right)^\frac{2p}{|p-2|} \cdot\tfrac{n-1}{2}$.
\end{theorem}


\subsection*{Acknowledgments} I am grateful to Keith Ball, Assaf Naor and Pierre Youssef for insightful discussions and useful feedback.


\section{Proof of Theorem \ref{thm:main}}

We say that a normed space $(E,\|\cdot\|_E)$ has Rademacher type $p$ if there exists a universal constant $T\in(0,\infty)$ such that for every $n\in\N$ and every $x_1,\ldots,x_n\in E$,
\begin{equation} \label{eq:type}
\frac{1}{2^n} \sum_{\e\in\{-1,1\}^n} \Big\| \sum_{i=1}^n \e_i x_i\Big\|_E^p \leq T^p \sum_{i=1}^n \|x_i\|_E^p.
\end{equation} 
The least constant $T$ such that \eqref{eq:type} is satisfied is denoted by $T_p(E)$. A standard symmetrization argument (see \cite[Proposition~9.11]{LT91}) shows that if $X_1,\ldots,X_n$ are independent $E$-valued random variables with $\mb{E}[X_i]=0$ for every $i\in\{1,\ldots,n\}$, then
\begin{equation} \label{eq:general-type}
\mb{E}\Big\|\sum_{i=1}^n X_i\Big\|_E^p \leq \big(2T_p(E)\big)^p \sum_{i=1}^n \mb{E} \|X_i\|_E^p.
\end{equation}

\subsection{Maurey's empirical method and its algorithmic counterparts} \label{sec:maurey}

 A classical theorem of Carath\'eodory asserts than if $\ms{T}$ is a subset of $\R^m$, then any point $z\in\mathrm{conv}(\ms{T})$ can be expressed as a convex combination of at most $m+1$ points of $\ms{T}$. Maurey's empirical method is a powerful dimension-free approximate version of Carath\'eodory's theorem, first popularized in \cite{Pis81}, that has numerous applications in geometry and theoretical computer science. Let $(E,\|\cdot\|_E)$ be a Banach space, consider a bounded subset $\ms{T}$ of $E$ and fix $z\in\mathrm{conv}(\ms{T})$. Since $z$ is a convex combination of elements of $\ms{T}$, there exists $m\in\N$, $\lambda_1,\ldots,\lambda_m\in(0,\infty)$ and $t_1,\ldots,t_m\in\ms{T}$ such that
\begin{equation} \label{eq:conv-comb}
\sum_{k=1}^m \lambda_k = 1 \qquad \mbox{and} \qquad z=\sum_{k=1}^m \lambda_kt_k.
\end{equation}
Let $X$ be an $E$-valued discrete random variable with $\mb{P}\{X=t_k\}=\lambda_k$ for all $k\in\{1,\ldots,m\}$ and consider $X_1,\ldots,X_d$ i.i.d.~copies of $X$. Then, conditions \eqref{eq:conv-comb} ensure that $X$ is well defined and $\mb{E}[X]=z$. Therefore, applying the Rademacher type condition \eqref{eq:general-type} to the centered random variables $\{X_s-z\}_{s=1}^d$ and normalizing, we get
\begin{equation}
\mb{E}\Big\| \frac{1}{d} \sum_{s=1}^d X_s - z\Big\|_E^p \leq \frac{(2T_p(E))^p}{d^{p-1}} \ \mb{E}\|X-z\|_E^p.
\end{equation}
Since $X$ takes values in $\ms{T}$, if $\ms{T} \subseteq R{\bf B}_E$, we then deduce that there exist $x_1,\ldots,x_d\in\ms{T}$ such that
\begin{equation}
\Big\|\frac{1}{d}\sum_{s=1}^d x_s - z\Big\|_E \leq \frac{4RT_p(E)}{d^{1-1/p}}.
\end{equation}

While the above argument is probabilistic, recent works have focused on derandomizing Maurey's sampling lemma for smaller classes of Banach spaces, thus constructing deterministic algorithms which output the empirical approximation $\tfrac{x_1+\ldots+x_d}{d}$ of $z$. The first result in this direction is due to Barman \cite{Bar18} who treated the case that $E$ is an $L_r(\mu)$-space, $r\in(1,\infty)$. This assumption was recently generalized by Ivanov in \cite{Iva21} who built a greedy algorithm which constructs the desired empirical mean in an arbitrary $p$-uniformly smooth space.


\subsection{Dimension reduction in $L_p(\mu)$ for uniformly bounded vectors} 
With Maurey's empirical method at hand, we are ready to proceed to the first part of the proof of Theorem \ref{thm:main}, namely the $\e$-isometric dimension reduction property of $L_p(\mu)$ under the strong assumption that the given point set consists of functions which are bounded in $L_\infty(\mu)$.

\begin{proposition} \label{prop:bounded}
Fix $p\in[1,\infty)$, $n\in\N$ and let $\{x_i\}_{i=1}^n$ be a family of vectors in $L_p(\mu)$ for some probability measure $\mu$. Denote by $L=\max_{i\in\{1,\ldots,n\}} \|x_i\|_{L_\infty(\mu)}\in[0,\infty]$. Then for every $\e\in(0,1)$, there exists $d\in\N$ with $d\leq\tfrac{32e^2(2L)^{2p}\log n}{\e^2}$ and \mbox{$y_1,\ldots,y_n\in\ell_p^d$ such that}
\begin{equation}
\forall \ i,j\in\{1,\ldots,n\}, \qquad \|x_i-x_j\|^p_{L_p(\mu)}- \e \leq \|y_i-y_j\|_{\ell_p^d}^p  \leq \|x_i-x_j\|^p_{L_p(\mu)}+\e.
\end{equation}
\end{proposition}

\begin{proof}
We shall identify $\ell_\infty^{\binom{n}{2}}$ with the vector space of all symmetric $n\times n$ real matrices with $0$ on the diagonal equipped with the supremum norm. Consider the set
\begin{equation}
\mathscr{C}_p = \big\{ \big( \|z_i-z_j\|_{L_p(\rho)}^p\big)_{i,j=1,\ldots,n}: \ \rho \mbox{ is a probability measure and } z_1,\ldots,z_n\in L_p(\rho)\big\} \subseteq \ell_\infty^{\binom{n}{2}}.
\end{equation}
It is obvious that $\ms{C}_p$ is a cone in the sense that $\ms{C}_p = \lambda \ms{C}_p$ for every $\lambda>0$ but moreover $\ms{C}_p$ is convex. To see this, consider $A,B\in\ms{C}_p$, probability spaces $(\Omega_1,\rho_1), (\Omega_2,\rho_2)$ and vectors $\{z_i\}_{i=1}^n, \{w_i\}_{i=1}^n$ in $L_p(\rho_1)$ and $L_p(\rho_2)$ respectively such that
\begin{equation} \label{eq:convex-cone}
\forall \ i,j\in\{1,\ldots,n\}, \qquad A_{ij} = \|z_i-z_j\|_{L_p(\rho_1)}^p \ \ \mbox{and} \ \ B_{ij} = \|w_i-w_j\|_{L_p(\rho_2)}^p.
\end{equation}
Fix $\lambda\in(0,1)$ and consider the disjoint union $\Omega_1\sqcup\Omega_2$ of $\Omega_1$ and $\Omega_2$ equipped with the probability measure $\rho(\lambda) = \lambda \rho_1+(1-\lambda)\rho_2$. Then, by \eqref{eq:convex-cone} the functions $\zeta_i:\Omega_1\sqcup\Omega_2\to\R$ given by $\zeta_i|_{\Omega_1} = z_i$ and $\zeta_i|_{\Omega_2}=w_i$, where $i\in\{1,\ldots,n\}$, belong in $L_p(\rho(\lambda))$ and satisfy the conditions
\begin{equation}
\forall \ i,j\in\{1,\ldots,n\}, \quad \|\zeta_i-\zeta_j\|_{L_p(\rho(\lambda))}^p = \lambda \|z_i-z_j\|_{L_p(\rho_1)}^p + (1-\lambda)  \|w_i-w_j\|_{L_p(\rho_2)}^p = \lambda A_{ij} + (1-\lambda) B_{ij},
\end{equation}
which ensure that $\lambda A+(1-\lambda)B\in\ms{C}_p$, making $\ms{C}_p$ a convex cone. Consider the embedding $\ms{M}:L_p(\mu)^n\to\ms{C}_p$ mapping a vector $z=(z_1,\ldots,z_n)$ to the corresponding distance matrix, i.e.
\begin{equation}
\forall \ i,j\in\{1,\ldots,n\}, \qquad \ms{M}(z)_{ij} = \|z_i-z_j\|_{L_p(\mu)}^p.
\end{equation}
Without loss of generality we will assume that the given points $x_1,\ldots,x_n\in L_p(\mu)$ are simple functions with  $\|x_i\|_{L_\infty(\mu)} \leq L$. Let $\{S_1,\ldots,S_m\}$ be a partition of the underlying measure space such that each $x_i$ is constant on each $S_k$ and suppose that $x_i|_{S_k} = a(i,k) \in[-L,L]$ for $i\in\{1,\ldots,n\}$ and $k\in\{1,\ldots,m\}$. Then, for every $i,j\in\{1,\ldots,n\}$, we have
\begin{equation} \label{eq:it-is-conv}
\ms{M}(x)_{ij} = \sum_{k=1}^m \int_{S_k} |x_i-x_j|^p \,\diff\mu = \sum_{k=1}^m \mu(S_k) \cdot \big|a(i,k)-a(j,k)\big|^p = \sum_{k=1}^m \mu(S_k) \ \ms{M}\big(y(k)\big)_{ij},
\end{equation}
where $y(k) \eqdef (a(1,k),\ldots,a(n,k))\in L_p(\mu)^n$ is a vector whose components are constant functions. As $\mu$ is a probability measure and $\{S_1,\ldots,S_m\}$ is a partition, identity \eqref{eq:it-is-conv} implies that
\begin{equation}
\ms{M}(x) \in \mathrm{conv} \big\{ \ms{M}\big( y(k)\big): \ k\in\{1,\ldots,m\}\big\} \subseteq \ell_\infty^{\binom{n}{2}}.
\end{equation}
Observe that since $a(i,k)\in[-L,L]$ for every $i\in\{1,\ldots,n\}$ and $k\in\{1,\ldots,m\}$, we have
\begin{equation} \label{eq:proof-diam}
\forall \ k\in\{1,\ldots,m\},\qquad \big\| \ms{M}\big(y(k)\big)\big\|_{\ell_\infty^{\binom{n}{2}}} = \max_{i,j\in\{1,\ldots,n\}} \big|a(i,k)-a(j,k)\big|^p \leq (2L)^p.
\end{equation}
Moreover, $\ell_\infty^{\binom{n}{2}}$ is $e$-isomorphic to $\ell_{p_n}^{\binom{n}{2}}$ where $p_n=\log\binom{n}{2}$. It is well-known (see \cite[Chapter~9]{LT91}) that $T_2(\ell_p) \leq \sqrt{p-1}$ for every $p\geq2$ and thus  
\begin{equation} \label{eq:proof-type}
T_2\big(\ell_\infty^{\binom{n}{2}} \big) \leq e\sqrt{p_n-1} < \sqrt{2e^2\log n}.
\end{equation}
Applying Maurey's sampling lemma (Section \ref{sec:maurey}) while taking into account \eqref{eq:proof-diam} and \eqref{eq:proof-type}, we deduce that for every $d\geq1$ there exist $k_1,\ldots,k_d\in\{1,\ldots,m\}$ such that
\begin{equation}
\Big\|\frac{1}{d} \sum_{s=1}^d \ms{M}\big( y(k_s)\big) - \ms{M}(x)\Big\|_{\ell_\infty^{\binom{n}{2}}} \leq \frac{2^{p+\frac{5}{2}}eL^p\sqrt{\log n}}{\sqrt{d}}.
\end{equation}
Therefore, if $\e\in(0,1)$ is such that $d\geq \tfrac{32e^2 (2L)^{2p}\log n}{\e^2}$ we then have
\begin{equation} \label{eq:use-maurey}
\forall \ i,j\in\{1,\ldots,n\},\qquad \Big| \frac{1}{d} \sum_{s=1}^d \big|a(i,k_s)-a(j,k_s)\big|^p - \|x_i-x_j\|_{L_p(\mu)}^p \Big| \leq \e.
\end{equation}e
Finally, consider for each $i\in\{1,\ldots,n\}$ a vector $y_i=(y_i(1),\ldots,y_i(d))\in\ell_p^d$ given by
\begin{equation}
\forall \ s\in\{1,\ldots,d\}, \qquad y_i(s) = \frac{a(i,k_s)}{d^{1/p}}
\end{equation}
and notice that \eqref{eq:use-maurey} can be equivalently rewritten as
\begin{equation}
\forall \ i,j\in\{1,\ldots,n\}, \qquad \|x_i-x_j\|^p_{L_p(\mu)}- \e \leq \|y_i-y_j\|_{\ell_p^d}^p  \leq \|x_i-x_j\|^p_{L_p(\mu)}+\e,
\end{equation}
concluding the proof of the proposition.
\end{proof}

\begin{remark}
It is worth emphasizing that the coordinates of the vectors $y_1,\ldots,y_n$ produced in Proposition \ref{prop:bounded} consist (up to rescaling) of values of the functions $x_1,\ldots,x_n$. Such low-dimensional embeddings via \emph{sampling} are a central object of study in approximation theory, see e.g.~the recent survey \cite{KKLT21} and the references therein.
\end{remark}

The additive version of the Johnson--Lindenstrauss lemma, first observed in \cite{LMPV17} as a consequence of a deep matrix deviation inequality (see also \cite[Chapter~9]{Ver18}), asserts that for every $n$ points $x_1,\ldots,x_n$ in a Hilbert space $\ms{H}$ and every $\e\in(0,1)$, there exists $d\leq \tfrac{C\log n}{\e^2}$ and points $y_1,\ldots,y_n\in\ell_2^d$ such that
\begin{equation} \label{eq:add-JL}
\forall \ i,j\in\{1,\ldots,n\}, \qquad \|x_i-x_j\|_{\ms{H}}-\e \leq \|y_i-y_j\|_{\ell_2^d} \leq \|x_i-x_j\|_\ms{H}+\e,
\end{equation}
where $C\in(0,\infty)$ is a universal constant. We will now observe that the spherical symmetry of ${\bf B}_{\ell_2}$ allows us to deduce a similar conclusion for points in ${\bf B}_{\ms{H}}$ by removing the incompressibility assumption from Proposition \ref{prop:bounded} when $p=2$. We shall use the standard notation $L_p^N$ for the space $L_p(\mu_N)$ where $\mu_N$ is the normalized counting measure on the finite set $\{1,\ldots,N\}$, that is
\begin{equation}
\forall \ a=(a_1,\ldots,a_N)\in\R^N,\qquad \|a\|_{L_p^N} \eqdef \Big(\frac{1}{N}\sum_{i=1}^N |a_i|^p\Big)^{1/p}.
\end{equation}
Observe that for $0<p<q\leq\infty$, we have ${\bf B}_{L_q^N} \subseteq {\bf B}_{L_p^N}$. 

\begin{corollary} \label{cor:l2}
There exists a universal constant $C\in(0,\infty)$ such that the following statement holds. Fix $n\in\N$ and let $\{x_i\}_{i=1}^n$ be a family of vectors in ${\bf B}_{\ms{H}}$ for some Hilbert space $\ms{H}$. Then for every $\e\in(0,1)$, there exists $d\in\N$ with $d\leq\tfrac{C(\log n)^3}{\e^4}$ and points $y_1,\ldots,y_n\in\ell_2^d$ such that
\begin{equation} \label{eq:l2l2}
\forall \ i,j\in\{1,\ldots,n\}, \qquad \|x_i-x_j\|_{\ms{H}}- \e \leq \|y_i-y_j\|_{\ell_2^d}  \leq \|x_i-x_j\|_{\ms{H}}+\e.
\end{equation}
\end{corollary}

Before proceeding to the derivation of \eqref{eq:l2l2} we emphasize that since the given points $\{x_i\}_{i=1}^n$ belong in ${\bf B}_\ms{H}$, Corollary \ref{cor:l2} is formally weaker than the Johnson--Lindenstrauss lemma. However we include it here since it differs from \cite{JL84} in that the low-dimensional point set $\{y_i\}_{i=1}^n$ is not obtained as an image of $\{x_i\}_{i=1}^n$ under a typical low-rank matrix from a specific ensemble.

\begin{proof} [Proof of Corollary \ref{cor:l2}]
Since any $n$-point subset $\{x_1,\ldots,x_n\}$ of $\ms{H}$ embeds linearly and isometrically in $L_2^n$, we assume that $x_1,\ldots,x_n\in {\bf B}_{L_2^{n}}$. We will need the following claim. 

\smallskip

\noindent {\it Claim.} Suppose that $X_1,\ldots,X_n$ are (not necessarily independent) random vectors, each uniformly distributed on the unit sphere $\mb{S}^{n-1}$ of $L_2^{n}$. Then, for some universal constant $S\in(0,\infty)$,
\begin{equation}
\mb{E} \big[ \max_{i\in\{1,\ldots,n\}} \|X_i\|_{L_\infty^{n}} \big] \leq S \sqrt{\log n},
\end{equation}

\begin{proof} [Proof of the Claim]

By a standard estimate of Schechtman and Zinn \cite[Theorem~3]{SZ90}, for a uniformly distributed random vector $X$ on the unit sphere $\mb{S}^{n-1}$ of $L_2^{n}$, we have
\begin{equation} \label{eq:sz}
\forall \ t\geq \gamma_1\sqrt{\log n}, \qquad \mb{P}\big\{ \|X\|_{L_\infty^{n}} > t\big\} \leq e^{-\gamma_2 t^2}
\end{equation}
for some absolute constants $\gamma_1,\gamma_2\in(0,\infty)$. Let $W\eqdef \max_{i\in\{1,\ldots,n\}} \|X_i\|_{L_\infty^{n}}$ and notice that
\begin{equation} \label{eq:break-integral}
\forall \ K\in(\gamma_1,\infty), \qquad \mb{E}[W] = \int_0^\infty \mb{P}\{W>t\} \,\diff t \leq K\sqrt{\log n} + \int_{K\sqrt{\log n}}^\infty \mb{P}\{W>t\} \,\diff t.
\end{equation}
By the union bound, we have
\begin{equation} \label{eq:union}
\forall \ t>0, \qquad \mb{P}\{W>t\} \leq \sum_{i=1}^n \mb{P}\{X_i>t\} = n \mb{P}\{X_1>t\}.
\end{equation}
Combining \eqref{eq:break-integral} and \eqref{eq:union}, we therefore get
\begin{equation}
\begin{split}
\mb{E}[W] \leq K&\sqrt{\log n}  +n \int_{K\sqrt{\log n}} \mb{P}\{X_1>t\}\,\diff t  \stackrel{\eqref{eq:sz}}{\leq} K\sqrt{\log n}  + n \int_{K\sqrt{\log n}}^\infty e^{-\gamma_2t^2} \,\diff t
\\ & = K\sqrt{\log n}  +n\sqrt{\log n} \int_{K}^\infty n^{-\gamma_2u^2} \,\diff u = K\sqrt{\log n} + \sqrt{\log n} \int_K^\infty n^{1-\gamma_2u^2} \,\diff u .
\end{split}
\end{equation}
Choosing $K>\gamma_1$ such that $K^2\gamma_2>1$, the exponent in the last integrand becomes negative, thus
\begin{equation}
\mb{E}[W]\leq K\sqrt{\log n}+2\sqrt{\log n} \int_K^\infty 2^{-\gamma_2 u^2}\,\diff u \leq S\sqrt{\log n}\end{equation}
for a large enough constant $S\in(0,\infty)$ and the claim follows.
\end{proof}

Now let $U \in \ms{O}(n)$ be a uniformly chosen random rotation on $\R^{n}$. The aforementioned claim shows that since $\|x_i\|_{L_2^{n}}\leq1$ for every $i\in\{1,\ldots,n\}$, writing $\hat{x}_i =\tfrac{x_i}{\|x_i\|_{L_2^{n}}}$, we have the estimate
\begin{equation} \label{eq:exp-rot}
\mb{E}\big[ \max_{i\in\{1,\ldots,n\}} \|Ux_i\|_{L_\infty^{n}}\big] \leq \mb{E}\big[ \max_{i\in\{1,\ldots,n\}} \|U\hat{x}_i\|_{L_\infty^{n}}\big] \leq S\sqrt{\log n}.
\end{equation}
Therefore, by \eqref{eq:exp-rot} and Proposition \ref{prop:bounded} there exists a constant $C\in(0,\infty)$ and a rotation $U\in\ms{O}(n)$ such that for every \mbox{$\e\in(0,1)$} there exists $d\leq \tfrac{C(\log n)^3}{\e^4}$ and points $y_1,\ldots,y_n\in\ell_2^d$ for which 
\begin{equation}
\forall \ i,j\in\{1,\ldots,n\}, \qquad  \|Ux_i-Ux_j\|^2_{L_2^{n}}- \e^2  \leq \|y_i-y_j\|_{\ell_2^d}^2  \leq \|Ux_i-Ux_j\|^2_{L_2^{n}}+\e^2.
\end{equation}
Since $\|Ua-Ub\|_{L_2^{n}} = \|a-b\|_{L_2^{n}}$ for every $a,b\in L_2^n$, the conclusion follows by the elementary inequality $|\alpha-\beta| \leq \sqrt{|\alpha^2-\beta^2|}$ which holds for every positive numbers $\alpha,\beta\in(0,\infty)$.
\end{proof}

\begin{remark} \label{rem:random}
Fix $p\in[1,\infty)$. The isometric embedding theorem of Ball \cite{Bal90} asserts that any $n$-point subset of $\ell_p$ admits an isometric embedding into $\ell_p^N$ where $N=\binom{n}{2}+1$. Suppose, more generally, that $n,N\in\N$ are such that $N$ is polynomial in $n$. Considerations in the spirit of the proof of Corollary \ref{cor:l2} (e.g. relying on \cite{SZ90}) then show that if $x_1,\ldots,x_n$ are independent uniformly random points in ${\bf B}_{L_p^N}$, then the random set $\{x_1,\ldots,x_n\}$ is $O(\log n)^{1/p}$-incompressible. In other words, incompressibility is a \emph{generic} property of random $n$-point subsets of ${\bf B}_{L_p^N}$. On the other hand, a typical $n$-point subset of ${\bf B}_{L_p^N}$ is known to be approximately a simplex due to work of Arias-de-Reyna, Ball and Villa \cite{AdRBV98} and so, in particular, \mbox{it can be bi-Lipschitzly embedded in $O(\log n)$ dimensions.}
\end{remark}

\subsection{Factorization and proof of Theorem \ref{thm:main}} \label{sec:fact}
Observe that Proposition \ref{prop:bounded} is rather non-canonical as the conclusion depends on the pairwise distances between the points $\{x_i\}_{i=1}^n$ in $L_p(\mu)$ whereas the bound on the dimension depends on $L=\max_i \|x_i\|_{L_\infty(\mu)}$. In order to deduce Theorem \ref{thm:main} from this (a priori weaker) statement we shall leverage the fact that Proposition \ref{prop:bounded} holds for \emph{any} probability measure $\mu$ by optimizing this parameter $L$ over all lattice-isomorphic images of $\{x_i\}_{i=1}^n$. The optimal such \emph{change of measure} which allows us to replace $L$ by $\|\max_i |x_i|\|_{L_p(\mu)}$ is a special case of a classical factorization theorem of Maurey (see \cite{Mau74} or \cite[Theorem~5]{JS01} for the general statement), whose short proof we include for completeness.

\begin{proposition} \label{thm:maurey}
Fix $n\in\N$, $p\in(0,\infty)$ and a probability space $(\Omega,\mu)$. For every points $x_1,\ldots,x_n\in L_p(\mu)$, there exists a nonnegative density function $f:\Omega\to\R_+$ supported on the support of $\max_i|x_i|$ such that if $\nu$ is the probability measure on $\Omega$ given by $\tfrac{\diff\nu}{\diff\mu}=f$, then
\begin{equation} \label{eq:fact}
\max_{i\in\{1,\ldots,n\}}\big\|x_i f^{-1/p}\big\|_{L_\infty(\nu)} \leq \big\| \max_{i\in\{1,\ldots,n\}} |x_i|\big\|_{L_p(\mu)}.
\end{equation}
\end{proposition}

\begin{proof}
Let $V=\mathrm{supp}(\max_i |x_i|)\subseteq\Omega$ and define the change of measure $f$ as
\begin{equation} \label{eq:cOm}
\forall \ \omega\in \Omega, \qquad f(\omega) \eqdef \frac{\max_{i\in\{1,\ldots,n\}}|x_i(\omega)|^p}{\int_\Omega \max_{i\in\{1,\ldots,n\}}|x_i(\theta)|^p\,\diff \theta}.
\end{equation}
Then, \eqref{eq:fact} is elementary to check.
\end{proof}

We are now ready to complete the proof of Theorem \ref{thm:main}.

\begin{proof} [Proof of Theorem \ref{thm:main}]
Fix a $K$-incompressible family of vectors $x_1,\ldots,x_n\in L_p(\Omega,\mu)$ and let $V=\mathrm{supp}( \max_i |x_i|)\subseteq\Omega$. Denote by $f:\Omega\to\R_+$ the change of density from Proposition \ref{thm:maurey}. If $\tfrac{\diff\nu}{\diff\mu}=f$, then the linear operator $T:L_p(V,\mu)\to L_p(\Omega,\nu)$ given by $Tg = f^{-1/p}g$ is (trivially) a linear isometry. Therefore, Proposition \ref{prop:bounded} and \eqref{eq:fact} show that there exists $d\in\N$ with $d\leq\tfrac{32e^2(2K)^{2p}\log n}{\e^2}$ and points $y_1,\ldots,y_n\in\ell_p^d$ such that the condition
\begin{equation}
\|x_i-x_j\|^p_{L_p(\mu)}- \e = \|Tx_i-Tx_j\|^p_{L_p(\nu)}- \e \leq \|y_i-y_j\|_{\ell_p^d}^p  \leq \|Tx_i-Tx_j\|^p_{L_p(\nu)} + \e = \|x_i-x_j\|^p_{L_p(\mu)}+\e,
\end{equation}
is satisfied for every $i,j\in\{1,\ldots,n\}$. This concludes the proof of Theorem \ref{thm:main}.
\end{proof}

\begin{remark}  \label{rem:mapislinear}
A careful inspection of the proof of Theorem \ref{thm:main} reveals that the low-dimensional points $\{y_i\}_{i=1}^n$ can be obtained as images of the given points $\{x_i\}_{i=1}^n$ under a linear transformation. Indeed, starting from a $K$-incompressible family of points $\{x_i\}_{i=1}^n$ in $L_p(\Omega,\mu)$, we use Proposition \ref{thm:maurey} to find a change of measure $T:L_p(V,\mu)\to L_p(\Omega,\nu)$ such that $\{Tx_i\}_{i=1}^n$ satisfy the stronger assumption of Proposition \ref{prop:bounded}. Then, for some $d\in\N$ with $d\leq\tfrac{32e^2(2K)^{2p}\log n}{\e^2}$ we find pairwise disjoint measurable subsets $S_1,\ldots,S_d$ of $\Omega$, each with positive measure, \mbox{such that if $S:L_p(\Omega,\nu)\to\ell_p^d$ is the linear map}
\begin{equation} \label{eq:operS}
\forall \ z\in L_p(\Omega,\nu), \qquad Sz \eqdef \frac{1}{d^{1/p}}\Big( \frac{1}{\mu(S_1)} \int_{S_1} z \,\diff\nu, \ldots , \frac{1}{\mu(S_d)} \int_{S_d} z\,\diff\nu\Big) \in\ell_p^d,
\end{equation}
then the points $\{y_i\}_{i=1}^n = \{(S\circ T)x_i\}_{i=1}^n\subseteq \ell_p^d$ satisfy the desired conclusion \eqref{eq:main}. 
\end{remark}

\noindent We conclude this section by observing that the argument leading to Theorem \ref{thm:main} is constructive.

\begin{corollary} \label{cor:alg}
In the setting of Theorem \ref{thm:main}, there exists a greedy algorithm which receives as input the high-dimensional points $\{x_i\}_{i=1}^n$ and produces as output the low-dimensional points $\{y_i\}_{i=1}^n$.
\end{corollary}

\begin{proof}
As the density \eqref{eq:cOm} is explicitly defined, the linear operator $T:L_p(V,\mu)\to L_p(\Omega,\nu)$ can also be efficiently constructed. On the other hand, in order to construct the operator $S$ defined by \eqref{eq:operS} one needs to find the corresponding partition $\{S_1,\ldots,S_d\}$ and this was achieved in Proposition \ref{prop:bounded} via an application of Maurey's sampling lemma to the cone $\ms{C}_p \subseteq \ell_\infty^{N}$ where $N=\binom{n}{2}$. As $\ell_\infty^{N}$ is $e$-isomorphic to the 2-uniformly smooth space $\ell_{\log N}^{N}$, Ivanov's result from \cite{Iva21} implies that the construction can be implemented by a greedy algorithm.
\end{proof}


\section{Proof of Theorem \ref{thm:no-Lp}}

In this section we prove Theorem \ref{thm:no-Lp}. The constructed subset of ${\bf B}_{\ell_p}$ which does not embed linearly into $\ell_p^d$ for small $d$ is a slight modification of the one considered in \cite{LMN05}.

\begin{proof} [Proof of Theorem \ref{thm:no-Lp}]
Fix $m\in\N$ and denote by $\{w_i\}_{i=1}^{2^m}$ the rows of the $2^m\times2^m$ Walsh matrix and by $\{e_i\}_{i=1}^{2^m}$ the coordinate basis vectors of $\R^{2^m}$. Consider the $n$-point set
\begin{equation}
\ms{S}_{n,p} = \{0\} \cup \{e_1,\ldots,e_{2^m}\} \cup \big\{ \tfrac{w_1}{2^{m/p}}, \ldots, \tfrac{w_{2^m}}{2^{m/p}}\big\} \subseteq {\bf B}_{\ell_p^{2^m}}
\end{equation}
where $n=2^{m+1}+1$ and suppose that $T:\ell_p^{2^m} \to \ell_p^d$ is a linear operator such that
\begin{equation} \label{eq:assT}
\forall \ x,y\in\ms{S}_{n,p},\qquad \omega(\|x-y\|_{\ell^{2^m}_p}) \leq \|Tx-Ty\|_{\ell_p^d} \leq \Omega(\|x-y\|_{\ell^{2^m}_p}).
\end{equation}
Assume first that $1\leq p<2$. If we write $w_i = \sum_{j=1}^{2^m} w_i(j) e_j$ then by orthogonality of $\{w_i\}_{i=1}^{2^m}$,
\begin{equation} \label{1}
\sum_{i=1}^{2^m} \|Tw_i\|_{\ell_2^d}^2 = \sum_{i=1}^{2^m} \Big\| \sum_{j=1}^{2^m} w_i(j) Te_j\Big\|_{\ell_2^d}^2 = \sum_{j,k=1}^{2^m} \langle w_j, w_k\rangle \langle Te_j, Te_k\rangle = 2^{m} \sum_{j=1}^{2^m} \|Te_j\|_{\ell_2^d}^2.
\end{equation}
By assumption \eqref{eq:assT} on $T$, we have
\begin{equation} \label{2}
\forall \ j\in\{1,\ldots,2^m\}, \qquad \|Te_j\|_{\ell_2^d}^2 \leq \|Te_j\|_{\ell_p^d}^2 \leq \Omega(1)^2
\end{equation}
and
\begin{equation} \label{3}
\forall \ j\in\{1,\ldots,2^m\}, \qquad \|Tw_j\|_{\ell_2^d}^2 \geq 2^{\frac{2m}{p}} d^{-\frac{2-p}{p}}\big\| T\big(\tfrac{w_j}{2^{m/p}} \big)\big\|_{\ell_p^d}^2 \geq 2^{\frac{2m}{p}} d^{-\frac{2-p}{p}} \omega(1)^2.
\end{equation}
Combining \eqref{1}, \eqref{2} and \eqref{3} we deduce that
\begin{equation}
2^{m(1+\frac{2}{p})} d^{-\frac{2-p}{p}} \omega(1)^2 \leq 4^m\Omega(1)^2,
\end{equation}
which is equivalent to $d\geq \left(\tfrac{\omega(1)}{\Omega(1)}\right)^\frac{2p}{2-p} 2^m = \left(\tfrac{\omega(1)}{\Omega(1)}\right)^\frac{2p}{|p-2|} \cdot\tfrac{n-1}{2}$. The case $p>2$ is treated similarly.
\end{proof}


\bibliographystyle{alpha}
\bibliography{quasi-dimension}

\end{document}